\documentclass[10pt,draft,reqno]{amsart}
     \makeatletter
     \def\section{\@startsection{section}{1}%
     \z@{.7\linespacing\@plus\linespacing}{.5\linespacing}%
     {\bfseries
     \centering
     }}
     \def\@secnumfont{\bfseries}
     \makeatother
\newtheorem{theorem}{Theorem}[section]
\newtheorem{lemma}[theorem]{Lemma}

\theoremstyle{definition}

\theoremstyle{remark}

\numberwithin{equation}{section}
\usepackage{tikz}

\usepackage{tikz-3dplot}

\usepackage{pgfplots}

\pgfplotsset{%
    compat=1.8,
    compat/show suggested version=false,
}

\usetikzlibrary{calc,fadings,decorations.pathreplacing}

\usepackage[many]{tcolorbox}

  \usepackage{wrapfig}

\usetikzlibrary{fadings}
\pgfdeclareverticalshading{up}{100bp}
{color(0bp)=(pgftransparent!0); color(50bp)=(pgftransparent!0);
 color(75bp)=(pgftransparent!25); color(100bp)=(pgftransparent!25)}
\pgfdeclarefading{up}{\pgfuseshading{up}}
\pgfdeclareradialshading{thick ring}{\pgfpointorigin}{
  color(0pt)=(pgftransparent!100); color(15pt)=(pgftransparent!100);
  color(20)=(pgftransparent!0);
  color(25bp)=(pgftransparent!100); color(50bp)=(pgftransparent!100)}
\pgfdeclarefading{thick ring}{\pgfuseshading{thick ring}}
\begin{tikzfadingfrompicture}[name=ball]
\shade [shading=ball, ball color=black] circle [radius=1];
\end{tikzfadingfrompicture}
\begin{tikzfadingfrompicture}[name=ring]
\fill [white, path fading=thick ring] circle [radius=1];
\fill [black, path fading=up, fading angle=215] circle [radius=1];
\end{tikzfadingfrompicture}

\tikzset{shiny ball/.style={
  fill=none, draw=none, shading=ball, shading angle=-15,
  postaction={fill=white, path fading=ball, opacity=0.75, fading angle=45},
  postaction={fill=white, path fading=ring}
}}

  \tikzstyle{mybox} = [draw=black, fill=black!20, very thick,
    rectangle, rounded corners, inner sep=10pt, inner ysep=20pt]
\tikzstyle{fancytitle} =[fill=black, text=white]

\tikzset{%
  >=latex, 
  inner sep=0pt,%
  outer sep=2pt,%
  mark coordinate/.style={inner sep=0pt,outer sep=0pt,minimum size=3pt,
    fill=black,circle}%
}

\newcommand{\mbr}{{\mathbb R}}

\newcommand{\cl}{{\mathcal L}}

\newcommand{\ovs}{\overline{\sigma}}

\newcommand{\la}{{\langle}}

\newcommand{\ra}{{\rangle}}

\newcommand{\norm}[1]{\left\lVert#1\right\rVert}

\begin{document}
\title[Limiting Means for Spherical Slices]{Limiting Means for Spherical Slices}

\author[Amy Peterson]{Amy Peterson*}
\thanks{* Corresponding author}
\address{\hskip -.05in  Amy Peterson: Department of Mathematics, University of Connecticut, Storrs, CT 062569, USA}
\email{amy.peterson@uconn.edu}

\author[Ambar N. Sengupta] {Ambar N.~Sengupta}
 \address{ \hskip -.05in  Ambar N. Sengupta: Department of Mathematics \\
  University  of Connecticut\\
Storrs, CT 062569}
\email{ambarnsg@gmail.com}




\begin{abstract}
We show that for a suitable class of functions of finitely-many variables, the limit of integrals along slices of a high dimensional sphere is a Gaussian integral on a corresponding finite-codimension affine subspace in infinite dimensions.
 \end{abstract}

\maketitle

 \section{Introduction}\label{s:intro}

 In this paper we generalize a result of \cite{SenGRL2018} showing that the large-$N$ limit of the integral of a function $f$ over affine slices of  a high dimensional sphere
$S^{N-1}(\sqrt{N})$  is Gaussian. In  \cite{SenGRL2018} this was proved for bounded $f$, and here we establish the result for $f$ in a suitable $L^p$ space, for any $p>1$.

 \subsection{Notation, Definitions, and Background}
 Let $A$ be a closed affine subspace of $l^2$ of finite codimension $m$ then there is
 $$
 Q: l^2 \to \mathbb{R}^m
 $$
 a continuous linear surjection and $w_0 \in \mathbb{R}^m$ so that we can write $A$ as a level set of $Q$,
 \begin{equation}\label{Aaslevelset}
 A= Q^{-1}(w_0).
 \end{equation}
 Let $z^0$ be the point on $Q^{-1}(w_0)$ closest to the origin and $S^{N-1}(\sqrt{N})$ be the sphere in $\mbr^N$ of radius $\sqrt{N}$ centered at the origin. Now we are interested in 'circles' formed by intersecting the part of $A$ in $\mbr^N$, $A_N$, with the sphere $S^{N-1}(\sqrt{N})$:
\begin{equation}\label{circle}
 S_{A_N} := A_N \cap S^{N-1}(\sqrt{N}).
\end{equation}
 Let
 $$
 Q_N : \mathbb{R}^N \to \mathbb{R}^m
 $$
defined by $Q_N = QJ_N$ where $J_N$ is the inclusion map from $\mathbb{R}^N$ to $l^2$. We can write $A_N = Q_N^{-1}(w_0)$ and the point on $A_N$ closest to the origin as $z^0_N$.
Thus the 'circle' $S_{A_N}$ is a sphere with center at $z^0_N$ and radius $a_{z^0_N} = \sqrt{ a^2 - |z^0_N|^2}$ where $a = \sqrt{N}$. See Figure \ref{circle1}.

Note that $z_N^0$ converges weakly to $z^0$ and so
\begin{eqnarray}\label{limitz}
Tz_N^0 \to Tz^0
\end{eqnarray}
for continuous linear $T: H \to X$ where $X$ is finite dimensional. See \cite{SenGRL2018} Proposition 4.1 for full detail.

 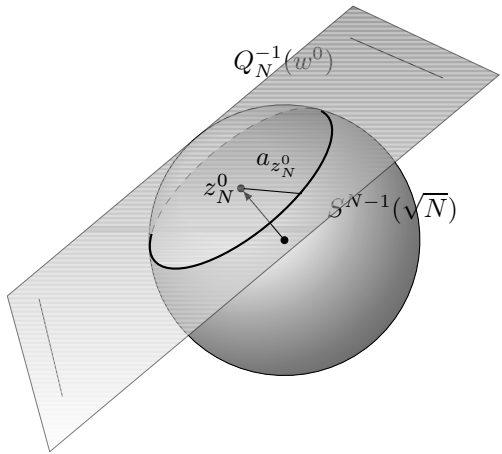
\begin{figure}\label{circle1}
\tikzset{
    partial ellipse/.style args={#1:#2:#3}{
        insert path={+ (#1:#3) arc (#1:#2:#3)}
 }}

 {
\centering

   \begin{tikzpicture}   [scale=.6,   rotate= 40]


\def\R{3} 
\def\angEl{45} 

\filldraw[ball color=white] (0,0) circle (\R);


 \coordinate [label=left:\textcolor{black}{$Q_N^{-1}(w^0)$}] (L) at (3.5, 2.25);

 \coordinate [label=left:\textcolor{black}{$S^{N-1}(\sqrt{N})$}] (X) at
 (3.5, -2);

  \coordinate [label=left:\textcolor{black}{$z_N^0$}] (z0) at
 (0,1.5);


 \coordinate [label=left:\textcolor{black}{${}$}](G) at (-6,.5);

\coordinate [label=left:\textcolor{black}{${}$}] (C) at (5,.5);

\coordinate (F) at (4,2.5);

\coordinate (H) at (-5 ,2.5);

 \draw (G)-- (H);

 \draw (C)--(F);


\shadedraw [color=black, opacity = .5]  (-7,-.25)--(6,-.25)--(4.5,3)--(-5.5,3)--(-7,-.25);


 \draw[->, shorten <=1pt,shorten >=1pt](0,0)--(0, 1.5);


\draw[black,fill=black] (0,0) circle (.5ex);

\draw[black,fill=black] (0,1.5) circle (.5ex);


 \shadedraw [color=black, dashed, opacity=0.3] (0,1.5) ellipse (2.5  and 1);

 \draw[thick] (0,1.5) [partial ellipse=150:370: 2.5 and 1];


 \draw[shorten <=1pt,shorten >=1pt](0,1.5)--(1, 0.5);


 \coordinate [label=left:\textcolor{black}{$a_{z_N^0}$}] (az0) at
 (1.4, 1 );

     \coordinate [label=left:\textcolor{black}{$z_N^0$}] (z0) at
 (0,1.5);


 (6, -3.5);

\end{tikzpicture}

}
\caption{\protect\raggedright The affine subspace $Q_N^{-1}(w^0)$ slices the sphere $S^{N-1}(\sqrt{N})$ in a `circle' with center $z_N^0$ and radius $a_{z_N^0}$.}
\end{figure}

We are interested in integrals of a function $\phi$ over $S_{A_N}$ with respect to the normalized surface area measure $\bar{\sigma}$:
\begin{eqnarray}\label{circint}
\int_{S_{A_N}} \phi(x_1,\ldots,x_k) d \bar{\sigma}(x) = \int_{S^{N-1}(\sqrt{N})\cap Q_N^{-1}(w^0)} \phi(x_1,\ldots, x_k) d \bar{\sigma}(x).
\end{eqnarray}
We will take $\phi$ to be a Borel function that only depends on the first $k$-coordinates for $k<N$. We will also need an important disintegration formula for the integral (\ref{circint}) and to that end we need the following projections.
Let
\begin{equation*}
P_{(k)}: l^2 \to \mathbb{R}^k = X \quad z \mapsto (z_1,\ldots,z_k)
\end{equation*}
be the projection from $l^2$ onto the first $k$-coordinates. Then let $\mathcal{L}$ be the restriction
of $P_{(k)}$ to $\ker Q$:
\begin{eqnarray}\label{cl}
\cl : \ker Q \to X
\end{eqnarray}
This is a surjection provided $\dim \left(\ker(Q) \right)> \dim(X)$. Further we define $\cl_N$ to be the restriction of $P_{(k)}$ to $\ker(Q_N)$:
\begin{eqnarray}\label{clN}
\cl_N : \ker Q_N =\ker Q \cap Z_N \to X
\end{eqnarray}
for large enough $N$, $\cl_{N}$ is also surjective. (See \cite{SenGRL2018} Proposition 6.2). Next we also want to restrict $\cl$ and $\cl_N$ to be isomorphisms. We define $\cl_{0}$ to be the restriction of $\cl$ to $\ker Q \ominus \ker\cl$ which is the orthogonal complement of $\ker Q \cap \ker\cl$ within $\ker Q$:
\begin{eqnarray*}\label{cl0}
\cl_{0}: \ker Q \ominus \ker\cl \to X.
\end{eqnarray*}
This is an isomorphism. Lastly let $\cl_{0,N}$ be the restriction of $\cl_N$ to $\ker Q_N \ominus \ker\cl_N$
$$
\cl_{0,N} : \ker Q_N \ominus \ker\cl_N \to X
$$
for large $N$ this is an isomorphism (See again \cite{SenGRL2018} Prop. 6.2).

Now we state without proof the disintegration formula we will need. See \cite{SenGRL2018} Theorem 3.3 for full detail.

\begin{theorem}\label{T:slicedisint}
Let $f$ be a bounded, or non-negative, Borel function defined on $S_{A_N} = S^{N-1}(\sqrt{N})\cap Q_N^{-1}(w^0)$ for some $w^0\in \mbr^m$.  Let ${z_N^0}$ be the point on  $Q_N^{-1}(w^0)$ closest to $0$. Let $\cl_{0,N}$ and $\cl_N$ be defined as above and let ${x^0}={\cl_N}({z_N^0})\in X$. Then
\begin{equation}\label{E:disintgenslice}
\begin{split}
&\int_{S^{N-1}(\sqrt{N})\cap Q^{-1}(w^0)}f\,d\sigma\\
&=\int_{x\in D_N}\left\{\int_{S_{A_N} \cap {\cl_N}^{-1}(x)}f\,d\sigma\right\}\frac{a_{{z_N^0}} } {\sqrt{a_{{z_N^0}}^2 -  \norm{{\cl_{0,N}}^{-1}(x-{x^0})}^2  }}\,\frac{dx}{|\det {\cl_{0,N}} |},
\end{split}
\end{equation}
where
$a_{{z_N^0}} = \sqrt{N - |z_N^0|^2}$ and $D_N$  consists of all $x\in x^0+{\cl_N}(\ker Q_N)\subset X$ for which the term under the square-root is positive:
\begin{equation}\label{E:defDLx0}
D_N=x^0+\{y\in {\cl_N}(\ker Q_N) :\,    \norm{{\cl_{0,N}}^{-1}(y)} < a_{{z_N^0}}\}.
\end{equation}
\end{theorem}

Taking $\phi$ to be a Borel function on $\mbr^k$ we let $f$ be the function obtained by extending $\phi$ to $\mbr^N$ by setting
$$
f(x) = \phi(x_1,\ldots, x_k) \quad \mbox{for all } x \in \mbr^N
$$
and denote $S^{N-1}(\sqrt{N})\cap Q^{-1}(w^0)\cap {\cl_N}^{-1}(x) = S_{A_N} \cap {\cl_N}^{-1}(x)$
then the disintegration formula (\ref{E:disintgenslice}) for this particular $f$ is:

\begin{equation}
\begin{split}
&\int_{S^{N-1}(\sqrt{N})\cap Q^{-1}(w^0)}f\,d\sigma\\
&=\int_{x\in D_N}\left\{\int_{S_{A_N}\cap {\cl_N}^{-1}(x)}f\,d\sigma\right\}\frac{a_{{z_N^0}} } {\sqrt{a_{{z_N^0}}^2 -  \norm{{\cl_{0,N}}^{-1}(x-{x^0})}^2  }}\,\frac{dx}{|\det {\cl_{0,N}} |} \\
&=\int_{x\in D_N}\left\{{\rm Vol}(S_{A_N}\cap {\cl_N}^{-1}(x) )\right\}\frac{a_{{z_N^0}} } {\sqrt{a_{{z_N^0}}^2 -  \norm{{\cl_{0,N}}^{-1}(x-{x^0})}^2  }}\,\frac{dx}{|\det {\cl_{0,N}} |}.
\end{split}
\end{equation}
Let $d= N-1$ then the volume of the sphere is:
\begin{equation}\label{E:volsphere} {\rm Vol}\left(S_{A_N}\cap {\cl_N}^{-1}(x) \right)
= c_{d-k-m} \left[a_{{{z_N^0}}}^2- \norm{{\cl_{0,N}}^{-1}(x-{x^0})}^2 \right]^{\frac{d-k-m}{2}}
\end{equation}
where $c_{d-k-m}$ is the surface measure of the $(d-k-m)$-dimensional sphere given, for all $j$, by the formula:
\begin{equation}\label{E:cjsurf}
c_{j}= 2\frac{\pi^{\frac{j+1}{2}}}{\Gamma\left(\frac{j+1}{2}\right)}.
\end{equation}
We can then rewrite (\ref{E:volsphere}) as
\begin{equation}\label{E:disintgenslice3b}
\begin{split}
 \int_{S^{N-1}(\sqrt{N})\cap Q^{-1}(w^0)}f\,d\sigma
&=c_{d-k-m}\int_{x\in D_N}
I_N(x)\,\frac{dx}{|\det {\cl_{0,N}} |},
\end{split}
\end{equation}
where
\begin{equation}\label{E:sliceIN}
I_N(x)= \phi(x)a_{z_N^0}    \left[a_{{{z_N^0}}}^2- \norm{{\cl_{0,N}}^{-1}(x-{x^0})}^2 \right]^{\frac{d-k-m-1}{2}}.
\end{equation}
The sphere $S^{N-1}(\sqrt{N})\cap Q^{-1}(w^0)$ has dimension $d-m$ and its volume is
$$c_{d-m}a_{z_N^0}^{d-m}.$$
So, using the normalized surface measure $\ovs$ on the sphere $S^{N-1}(\sqrt{N})\cap Q^{-1}(w^0)$, we have
\begin{equation}\label{E:disintgenslice3c}
\begin{split}
 \int_{S^{N-1}(\sqrt{N})\cap Q^{-1}(w^0)}f\,d\ovs
&=\frac{c_{d-k-m}}{c_{d-m}a_{z_N^0}^{d-m} } \int_{x\in D_N}
I_N(x)\,\frac{dx}{|\det {\cl_{0,N}} |},
\end{split}
\end{equation}
where $I_N(x)$ is as in (\ref{E:sliceIN}).

\subsection{Related literature}\label{ss:rl}  This paper is a generalization of work done in \cite{SenGRL2018} where more detailed results are proved for a bounded Borel function $\phi$.

The connection between Gaussian measure and the uniform measure on high dimensional spheres appeared originally in the works of Maxwell  \cite{MaxGas1860}  and Boltzmann  \cite[pages 549-553]{BoltzStud1868}. Later works included Wiener's paper \cite{WienDS1923} on ``differential space'', L\'evy \cite{LevyPLAF1922}, McKean \cite{McKGDS1973}, and Hida \cite{HidaStat1970}. The work of Mehler \cite{Mehler1866} is one example illustrating the classical interest in functions on high-dimensional spheres.

For the theory of Gaussian measures in infinite dimensions we refer to the monographs of Bogachev \cite{BogGM1998} and Kuo \cite{KuoGB1975}.

This paper is the fourth in a series of papers. The first \cite{HolSenGR2012} develops the Gaussian Radon transform  for Banach spaces, where a support theorem was established. The second \cite{SenGRL2016} establishes the result for hyperplanes and the third \cite{SenGRL2018} proves the result for the case of affine planes.

\section{Limiting Results}

In this section we review our previous results and prove the main result of this paper.

\subsection{Previous Results}
From the previous paper \cite{SenGRL2018} the main result was the following theorem:

 \begin{theorem}\label{T:RadnonNlim} Let $A$ be a finite-codimension closed affine subspace in $l^2$, specified by (\ref{Aaslevelset}). Let $k$ be a positive integer; suppose that the image of $A$ under the coordinate projection $l^2\to\mbr^k:z\mapsto z_{(k)}=(z_1,\ldots, z_k)$ is all of $\mbr^k$.  Let $\phi$ be a bounded Borel function on $\mbr^k$. Then
 \begin{equation}\label{E:limRfS}
 \lim_{N\to\infty}\int_{S_{A_N}} \phi(x_1,\ldots, x_k)\,d\ovs(x_1,\ldots, x_N)=\int_{\mbr^\infty}\phi(z_{(k)})  \,d\mu(z),  \end{equation}
 where $\ovs$ is the normalized surface area measure on $S_{A_N}$, and $\mu $ is the probability measure on $\mbr^\infty$ specified by the characteristic function
  \begin{equation}\label{E:cphi}
  \begin{split}
\int_{\mbr^\infty} \exp\left({i\la t, x\ra}\right)\,d\mu (x) &=\exp\left({i\la t, p_A  \ra-\frac{1}{2}\norm{P_0t}^2}\right)\qquad   \hbox{for all $t\in\mbr^\infty_0$,}
\end{split}
\end{equation}
where $p_A$ is the point on $A$ closest to the origin and $P_0$ is the orthogonal projection in $l^2$ onto the subspace $A-p_A$. \end{theorem}

We give a sketch of the proof here for full detail refer to \cite{SenGRL2018} Theorem 2.1.

 The pushforward measure ${\pi_{(k)}}_*\mu$ of $\mu$ to $\mbr^k$ is
 \begin{equation}
{ \pi_{(k)}}_*\mu(S)=\mu\bigl(\pi_{(k)}^{-1}(S)\bigr),\qquad\hbox{for all Borel $S\subset\mbr^k$,}
 \end{equation}
 where
 $$\pi_{(k)}:\mbr^\infty\to\mbr^k:z\mapsto z_{(k)}=(z_1,\ldots, z_k)$$
  is the projection on the first $k$ coordinates. Now define $\mu_\infty$ on $\mbr^k$ by
   \begin{equation}\label{E:muinfty2}
 d\mu_\infty(x)= (2\pi)^{-k/2} \exp\left({-\frac{1}{2}\la  (\cl_0\cl_0^*)^{-1}(x-z^0_{(k)}), x-z^0_{(k)} \ra }\right) \,\frac{dx}{|\det \cl_{0}|  },
 \end{equation}
 where $\cl_0$ is given in (\ref{cl0}) and $z^0_{(k)}$ is the first $k$-coordinates of $z^0$, the point on $A = Q^{-1}(w_0)$ closest to the origin.

From their respective characteristic functions we can deduce that
 $$
  \pi_{(k)} \star \mu(S) = \mu_\infty
  $$

Now, using this and Theorem \ref{T:limintfsig} below, we can conclude that

\begin{equation}\label{E:mainresu1}
\begin{split}
&\lim_{N\to\infty} \int_{S^{N-1}(\sqrt{N})\cap Q_N^{-1}(w^0)}\phi(x_1,\ldots, x_k)\,d{\bar\sigma}(x_1,\ldots, x_N) \\
&=   \int_{\mbr^k} \phi \,d \mu_{\infty}\\
&= \int_{\mbr^k} \phi \,d {\pi_{(k)}}_*\mu = \int_{\mbr^\infty} \phi\circ{\pi_{(k)}} \,d\mu.
\end{split}
\end{equation}

For the first equality in (\ref{E:mainresu1}) we need the following theorem (from \cite{SenGRL2018} Theorem 4.1).

\begin{theorem}\label{T:limintfsig} Let $A$ be an affine subspace of $l^2$ given by $Q^{-1}(w^0)$, where $Q:l^2\to \mbr^m$ is a continuous linear surjection.
 Suppose that the projection  $P_{(k)}: l^2\to \mbr^k: z\mapsto z_{(k)}$  maps $\ker Q$ onto $\mbr^k$. Let $S^{N-1}(\sqrt{N})$ be the sphere of radius $\sqrt{N}$ in the subspace $\mbr^N\oplus\{0\}$ in $l^2$.  Let $\phi$ be a bounded Borel function on $\mbr^k$ and let $f$ be the function obtained by extending $\phi$ to $l^2$ by setting
\begin{equation}\label{E:fphi}
f(x)=\phi(x_1,\ldots, x_k)\qquad\hbox{for all $x\in l^2$.}
\end{equation}
     Then
   \begin{equation}\label{E:disintgenslice0}
\begin{split}
&\lim_{N\to\infty} \int_{S_{Z_{N}}(\sqrt{N})\cap Q_N^{-1}(w^0)}f\,d{\bar\sigma}\\
&=(2\pi)^{-k/2} \int_{x\in \mbr^k}
\phi(x) \exp\left({- \frac{  \la ({\cl_0}{\cl_0}^*)^{-1}(x-{z^{0}}_{(k)}), x-{z^{0}}_{(k)} \ra}{2} }\right)\,  \frac{dx}{ |\det\cl_0|},
\end{split}
\end{equation}
where ${\cl_0}$ is the restriction of the  projection $P_{(k)}$ to $\ker Q\ominus \ker  P_{(k)}$, and $z^0$ is the point on  $ Q^{-1}(w^0)$ closest to the origin.
           \end{theorem}

Note that in order to extend  Theorem \ref{T:RadnonNlim} for a more general function $\phi$ we only need to extend Theorem \ref{T:limintfsig}.

Again we give a sketch of the proof for Theorem \ref{T:limintfsig} . Let $a = \sqrt{N}$ and $d = N-1$. From the disintegration formula above  (\ref{E:disintgenslice3c}) we have
\begin{equation}\label{E:disintgenslice17}
 \lim_{N\to\infty} \int_{S^{N-1}(\sqrt{N})\cap Q_N^{-1}(w^0)}f\,d{\bar\sigma}\\
 =  \lim_{N\to\infty}  \frac{  c_{d-k-m}}{a_{{z^{0,N}}}^{k}c_{d-m}} \int_{\mbr^k}
I_N\,\frac{dx}{|\det {\cl_{0,N}} |},
\end{equation}
where
\begin{equation}\label{E:IN3}
I_N=  \phi(x)
  \left\{	1 - a_{{z^{0,N}}}^{-2} \norm{{{\cl_{0,N}} }^{-1}(x-{z^{0,N}}_{(k)})}^2  \right\}^{\frac{d-k-m-1}{2}}1_{D_N}(x).\end{equation}
We state here the limits of the constant term outside the integral  (in (\ref{E:disintgenslice17})), as well as those of the full integrand on the right hand side, including the determinant term without proof, for full detail refer to \cite{SenGRL2018}:

\begin{equation}\label{E:constlim}
  \lim_{N \to \infty}\frac{ c_{d-k-m}}{a_{{z^{0,N}}}^{k}c_{d-m}} = (2\pi)^{-k/2},
\end{equation}

\begin{equation}\label{E:L0Ndet}
\lim_{N\to\infty} |\det {\cl_{0,N}} | =  \lim_{N\to\infty} \det({\cl_{0,N}}{\cl_{0,N}^*})  =  \det({\cl_0}{\cl_0^*}) = |\det{\cl_0}|,
\end{equation}
and
\begin{equation}\label{E:limN}
\begin{split}
\lim_{N \to \infty}\left\{1 - a_{{z^{0,N}}}^{-2} \norm{{{\cl_{0,N}} }^{-1}(x-{z^{0,N}}_{(k)})}^2  \right\}^{\frac{d-k-m-1}{2}}1_{D_N}(x) \\
 & \hskip -2 in   =\lim_{N\to\infty}  \left\{	1 - a_{{z^{0,N}}}^{-2} \norm{{{\cl_{0,N}} }^{-1}(x-{z^{0,N}}_{(k)})}^2  \right\}^{\frac{N}{2}} 1_{D_N}(x)\\
& \hskip -2 in =\exp\left(-\frac{1}{2}  \la ({\cl_0}{\cl_0}^*)^{-1}(x-z^0_{(k)}), x-z^0_{(k)}\ra\right).
 \end{split}
\end{equation}

Therefore if $\phi$ is such that we can apply dominated convergence theorem in (\ref{E:disintgenslice17}) we have,
\begin{align}\label{E:disintgenslice6}
 \lim_{N\to\infty} \int_{S^{N-1}(\sqrt{N})\cap Q_N^{-1}(w^0)}f\,d{\bar\sigma}
&=  \lim_{N\to\infty}  \frac{  c_{d-k-m}}{a_{{z^{0,N}}}^{k}c_{d-m}} \int_{\mbr^k}
I_N\,\frac{dx}{|\det {\cl_{0,N}} |} \\
& \hskip -1.5 in = (2\pi)^{-k/2} \int_{\mbr^k}\lim_{N\to\infty}
I_N\,\frac{dx}{|\det {\cl_{0,N}} |} \label{domconv}\\
&\hskip -1.5 in = (2\pi)^{-k/2}\int_{\mbr^k}\phi(x) \exp \left({-\frac{1}{2}\la  ({\cl_0}{\cl_0}^*)^{-1}(x-z^0_{(k)}), x-z^0_{(k)} \ra }\right) \,\frac{dx}{|\det{\cl_0}|  }\\
&\hskip -1.5 in =(2\pi)^{-k/2}\int_{\mbr^k}\phi(x) \exp \left({-\frac{1}{2}\norm{{\cl_0}^{-1}(x-z^0_{(k)})} }^2\right) \,\frac{dx}{|\det{\cl_0}|  }
\end{align}

which is the result in Theorem \ref{T:limintfsig}.

\subsection{The Main Result}
We turn now to the main result of this paper, an extension of the previous result Theorem \ref{T:RadnonNlim} to more general functions. We will show that if $\phi$ is a Borel function on $\mbr^k$ which is $L^p$, $p>1$, with respect to the Gaussian measure with density proportional to
$$e^{-\norm{\cl^{-1}_{0}(x-z^{0}_{(k)})}^2/2}\,dx,$$
then the conclusion of $\ref{T:RadnonNlim}$ still holds. To this end we state and prove a generalization of Theorem \ref{T:limintfsig}.

\begin{theorem}\label{T:limintfsig2} Let $A$ be an  affine subspace of $l^2$ given by $Q^{-1}(w^0)$, where $Q:l^2\to \mbr^m$ is a continuous linear surjection.  Suppose that the projection  $P_{(k)}: l^2\to \mbr^k: z\mapsto z_{(k)}$  maps $\ker Q$ onto $\mbr^k$. Let $S^{N-1}(\sqrt{N})$ be the sphere of radius $\sqrt{N}$ in the subspace $\mbr^N\oplus\{0\}$ in $l^2$.  Let $\phi$ be a Borel function on $\mbr^k$ which is in $L^p$ with respect to the Gaussian measure with density proportional to
$$e^{-\norm{\cl^{-1}_{0}(x-z^{0}_{(k)})}^2/2}\,dx,$$
for some $p>1$, and let $f$ be the function obtained by extending $\phi$ to $l^2$ by setting
\begin{equation}\label{E:fphi2}
f(x)=\phi(x_1,\ldots, x_k)\qquad\hbox{for all $x\in l^2$.}
\end{equation}
     Then
   \begin{equation}\label{E:disintgenslice02}
\begin{split}
&\lim_{N\to\infty} \int_{S^{N-1}(\sqrt{N})\cap Q_N^{-1}(w^0)}f\,d{\bar\sigma}\\
&=(2\pi)^{-k/2} \int_{x\in \mbr^k}
\phi(x) \exp\left({- \frac{  \la ({\cl_0}{\cl_0}^*)^{-1}(x-{z^{0}}_{(k)}), x-{z^{0}}_{(k)} \ra}{2} }\right)\,  \frac{dx}{ |\det\cl_0|},
\end{split}
\end{equation}
where ${\cl_0}$ is the restriction of the  projection $P_{(k)}$ to $\ker Q\ominus \ker  P_{(k)}$, and $z^0$ is the point on  $ Q^{-1}(w^0)$ closest to the origin.
\end{theorem}

\begin{proof}
Utilizing the proof from Theorem \ref{T:limintfsig} we need only show (\ref{domconv}) still holds, that is,
\begin{eqnarray}\label{e:domconv2}
\lim_{N\to\infty}  \frac{  c_{d-k-m}}{a_{{z^{0,N}}}^{k}c_{d-m}} \int_{\mbr^k}
I_N\,\frac{dx}{|\det {\cl_{0,N}} |}
= (2\pi)^{-k/2} \int_{\mbr^k}\lim_{N\to\infty}
I_N\,\frac{dx}{|\det {\cl_{0,N}} |}
\end{eqnarray}
where
$I_N = \phi(x)
  \left\{	1 - a_{{z^{0,N}}}^{-2} \norm{{{\cl_{0,N}} }^{-1}(x-{z^{0,N}}_{(k)})}^2  \right\}^{\frac{d-k-m-1}{2}}1_{D_N}(x)$
  and $D_N$ is all $x \in \mbr^k$ such that the square-root term is positive.

First we have the following inequality,
\begin{equation}
 \begin{split}
& \left\{	1 - a_{{z^{0,N}}}^{-2} \norm{{{\cl_{0,N}} }^{-1}(x-{z^{0,N}}_{(k)})}^2  \right\}^{d-k-m-1} \\
 &= \left\{	1 - \frac{\norm{{{\cl_{0,N}} }^{-1}(x-{z^{0,N}}_{(k)})}^2}{N-|z^{0,N}|^2}  \right\}^{N-k-m-2}
\\
&
\leq \left\{	1 - \frac{\norm{{{\cl_{0,N}} }^{-1}(x-{z^{0,N}}_{(k)})}^2}{N}  \right\}^{N-k-m-2}.
\end{split}
\end{equation}

We observe that, for $N>k+m+2$, the maximum of the function
$$
\left( 1 - \frac{y}{N}\right)^{N-k-m-2} e^y \quad \mbox{for all $y \in (0,N]$}
$$
occurs at $y=k+m+2$; this is seen by checking that the derivative $d/dy$ is positive for $y \in [0,k+m+2)$ and negative for $y \in (k+m+2,N]$. Thus,

$$
\left( 1 - \frac{y}{N}\right)^{N-k-m-2} e^y \leq \left( 1 - \frac{k+m+2}{N}\right)^{N-k-m-2} e^{k+m+2}
$$

Taking $y = \norm{{{\cl_{0,N}} }^{-1}(x-{z^{0,N}}_{(k)})}^2$, we have:

\begin{equation}
    \begin{split}
        &
   \left\{	1 - \frac{\norm{{{\cl_{0,N}} }^{-1}(x-{z^{0,N}}_{(k)})}^2}{N}  \right\}^{N-k-m-2} \\
   \leq& \left( 1 - \frac{k+m+2}{N}\right)^{N-k-m-2} e^{k+m+2}e^{-\norm{{\cl_{0,N}} ^{-1}(x-{z^{0,N}}_{(k)})}^2} \\
&\leq  e^{k+m+2}e^{-\norm{{{\cl_{0,N}} }^{-1}(x-{z^{0,N}}_{(k)})}^2}.
 \end{split}
\end{equation}
Thus
\begin{equation}\label{E:1minusexpbd}
\left\{	1 - \frac{\norm{{{\cl_{0,N}} }^{-1}(x-{z^{0,N}}_{(k)})}^2}{N}  \right\}^\frac{N-k-m-2}{2}
 \leq  e^\frac{k+m+2}{2}e^{\frac{- 1}{2}\norm{{\cl^{-1}_{0,N}} (x-{z^{0,N}}_{(k)})}^2}
\end{equation}

Lemma \ref{L:inequ} gives  the bound:
\begin{equation}\label{E:expbd}
e^{-\frac{1}{2}\norm{{\cl^{-1}_{0,N}} (x-{z^{0,N}}_{(k)})}^2} \leq e^{-\frac{1}{2}\norm{{\cl^{-1}_{0}} (x-{z^{0}}_{(k)})}^2}.
\end{equation}

Let
\begin{equation}
    \hbox{ $a_N(x) = \cl^{-1}_{0}(x-z^{0,N}_{(k)})$ and $a(x) = \cl^{-1}_{0}(x-z^{0}_{(k)})$.}
\end{equation}
Then by (\ref{limitz}), for any $\epsilon > 0$ and large enough $N$

$$
\norm{a_N(x) -a(x)} <\epsilon.
$$
Then
\begin{equation}
\begin{split}
    \norm{a(x)}^2 - \norm{a_N(x)}^2 &= \left(\norm{a(x)} - \norm{a_N(x)}\right)\left(\norm{a(x)} + \norm{a_N(x)}\right)
    \\
    & \leq \norm{a(x) - a_N(x)}\left(\norm{a(x)} + \norm{a_N(x)}\right)\\
    &\leq \epsilon\left( \norm{a(x)} + \norm{a_N(x) -a(x)} + \norm{a(x)}\right) \\
    &\leq \epsilon \left( 2 \norm{a(x)} + \epsilon\right) \\
    \end{split}
\end{equation}
and so
\begin{equation}
-\norm{a_N(x)}^2 \leq \epsilon^2 + 2\epsilon\norm{a(x)} - \norm{a(x)}^2.
\end{equation}
This gives us:
\begin{equation}
\begin{split}
e^{-\norm{a_N(x)}^2/2} &\leq e^{\epsilon^2/2 }e^{{\epsilon}\norm{a(x)}-\norm{a(x)}^2/2}.
  \end{split}
\end{equation}
Now since $\phi$ is a Borel function on $\mbr^k$ which is in $L^p$ with respect to the Gaussian measure with density proportional to
$$e^{-\norm{a(x)}^2/2}\,dx,$$
for some $p>1$.
We have then the bound
\begin{equation}\label{E:fdom}
|\phi(x)    e^{-\norm{a_N(x)}^2/2}| \leq  |{\phi}(x)|  e^{\epsilon^2/2 }e^{    {\epsilon}\norm{a(x)} } e^{-  \norm{a(x)}^2/2}\
\end{equation}
The dominating function is integrable:
\begin{equation}
    \begin{split}
         e^{\epsilon^2/2 }\int_{\mbr^k}{\phi}(x)e^{    {\epsilon}\norm{a(x)} } e^{-  \norm{a(x)}^2/2}\,dx
      \leq c_\epsilon\left\{\int_{\mbr^k}{\phi}(x)^pe^{-\norm{a(x)}^2/2}\,dx\right\}^{1/p},
    \end{split}
\end{equation}
where
\begin{equation}
    c_{\epsilon}
    =e^{\epsilon^2/2 }
    \left\{    \int_{\mbr^k}  e^{     {\epsilon}q\norm{a(x)} -\norm{a(x)}^2/2 }
    \right\}^{1/q}
\end{equation}
and $q$ is the conjugate to $p$ as usual: $p^{-1}+q^{-1}=1$. The integral in $c_{\epsilon} $ is finite because, after changing variables to $y=a(x)$,
\begin{equation}
\int_{\mbr^k}e^{t\norm{y}-\norm{y}^2/2}\,dy =|S^{k-1}|\int_0^{\infty}e^{tR-R^2/2}R^{k-1}\,dR<\infty,
\end{equation}
for any $t\in\mbr$.

Using the argument above we can conclude the dominated convergence in  (\ref{e:domconv2}) holds:
\begin{equation*}
\begin{split}
\lim_{N\to\infty}  \frac{  c_{d-k-m}}{a_{{z^{0,N}}}^{k}c_{d-m}} \int_{\mbr^k}
I_N(x) \,\frac{dx}{|\det {\cl_{0,N}} |} \\
& \hskip -2.25in = \frac{1}{2\pi^{k/2}} \int_{\mbr^k}\lim_{N\to\infty} \phi(x)  \left\{	1 - a_{{z^{0,N}}}^{-2} \norm{{{\cl_{0,N}} }^{-1}(x-{z^{0,N}}_{(k)})}^2  \right\}^{\frac{d-k-m-1}{2}}\,\frac{dx}{|\det {\cl_{0,N}} |}
\end{split}
\end{equation*}
and using the limits (\ref{E:constlim}), (\ref{E:L0Ndet}), and (\ref{E:limN}) we have established:
\begin{align*}
\lim_{N\to\infty} \int_{S^{N-1}(\sqrt{N})\cap Q_N^{-1}(w^0)}f\,d{\bar\sigma}
&=  \lim_{N\to\infty}  \frac{  c_{d-k-m}}{a_{{z^{0,N}}}^{k}c_{d-m}} \int_{\mbr^k}
I_N\,\frac{dx}{|\det {\cl_{0,N}} |} \\
&\hskip -1.75 in =\frac{1}{2\pi^{k/2}} \int_{\mbr^k}\lim_{N\to\infty} \phi(x)  \left\{	1 - a_{{z^{0,N}}}^{-2} \norm{{{\cl_{0,N}} }^{-1}(x-{z^{0,N}}_{(k)})}^2  \right\}^{\frac{d-k-m-1}{2}}\,\frac{dx}{|\det {\cl_{0,N}} |} \\
& \hskip -1.75 in = (2\pi)^{-k/2} \int_{x\in \mbr^k}
\phi(x) \exp\left({- \frac{  \la ({\cl_0}{\cl_0}^*)^{-1}(x-{z^{0}}_{(k)}), x-{z^{0}}_{(k)} \ra}{2} }\right)\,  \frac{dx}{ |\det\cl_0|}.
\end{align*}

\end{proof}

\begin{lemma}\label{L:inequ}
With notation as above,
\begin{equation}\label{E:inequ}
\norm{\cl^{-1}_{0,N}\left(x-z^{0,N}_{(k)}\right)} \geq \norm{\cl^{-1}_{0}\left(x-z^{0,N}_{(k)}\right)}
\end{equation}
\end{lemma}
\begin{proof}
Recall the definition of $\cl$, it is the projection on $k$ coordinates restricted to the $\ker Q$:
\begin{eqnarray}
\cl:\ker Q \to X = \mbr^k
\end{eqnarray}
and $\cl_{0}$ is the restriction of $\cl$ to the orthogonal complement of $\ker \cl$ inside $\ker Q$. Since $\cl$ is surjective $\cl_0$ is an isomorphism. Let $x \in \mbr^k$  and $y_0 = \cl_0^{-1}(x)$ then any vector $y \in \cl^{-1}(x)$ can be written as
$$
y = y_0 + v \quad v \in \ker \cl.
$$
This means
$$
\norm{y} = \norm{y_0 + v} \quad v \in \ker \cl
$$
for all $y \in \cl^{-1}(x)$
therefore $y_0 = \cl^{-1}_{0}(x)$ is the point in $\cl^{-1}(x)$ of smallest norm.
By the same argument, provided $N$ is large enough for $\cl_N$ to be a surjection, for $x \in \mbr^k$, the point $\cl^{-1}_{0,N}(x)$ is the point on $\cl^{-1}_N(x)$ of smallest norm.

Let $y \in \ker (Q_N)$ then $y \in \mbr^N$ and $Q_Ny = 0$. Taking $\mbr^N$ to be contained in $l^2$  as $\mbr^N \oplus \{0\}$ then for all $y \in \ker Q_N$ we take $y = (y,0)$. Now
$$
0= Q_Ny = Q(J_Ny)
$$
therefore $J_Ny \in \ker Q$ and so $(y,0) \in \ker Q$. Thus $\ker Q_N$ is contained in $\ker Q$.

Now for $y \in \ker Q_N$ we have $\cl (y) = \cl_N(y)$ since both $\cl$ and $\cl_N$ are the projection onto the first $k$ coordinates.  Since $\cl_N^{-1}(x)$ is all $y \in \ker Q_N$ such that $\cl_N (y) = x $ it is contained in $\cl^{-1}(x)$.

Now we have the inequality (\ref{E:inequ}).
\end{proof}

Let us look at an example that shows the necessity of the $L^p$, $p>1$, condition and the difficult nature of the limit of Gaussian integrals above. In this context for the function
$$g(x)=e^{x^2/2}(1+x^2)^{-1}\qquad\hbox{for all $x\in\mbr$,}$$
we have
\begin{equation}
\int_{\mbr}g(x)e^{-x^2/2}dx<\infty\qquad\hbox{but}\qquad \int_{\mbr}g(x)e^{-(x-x_N)^2/2}dx=\infty \quad\hbox{for all $x_N\neq 0$,}
\end{equation}
and so
\begin{equation}
\lim_{z\to 0}\int_{\mbr}g(x)e^{-(x-z)^2/2}dx \neq \int_{\mbr}g(x)e^{-x^2/2}dx.
\end{equation}

\par\bigskip\noindent
{\bf{ Acknowledgments.} }  We would like to thank Irfan Alam for many helpful discussions on the subject. We would also like to thank the online TikZ community and  Arthur Parzygnat for help with the figure.

\bibliographystyle{amsplain}

\end{document}